\documentclass[]{amsart}
\usepackage{amsmath}
\usepackage{amssymb}
\usepackage{amsthm}
\usepackage{amsfonts}
\usepackage[abbrev]{amsrefs}
\usepackage{bm}
\usepackage{graphicx}
\usepackage{enumerate}
\usepackage{url}
\usepackage{epsfig}
\usepackage{color}
\usepackage{float}
\usepackage{setspace}
\usepackage{comment}
\usepackage{appendix}
\usepackage[hang,small,bf]{caption}
\usepackage[subrefformat=parens]{subcaption}
\theoremstyle{definition}

\newtheorem{thm}{Theorem}[section]

\newtheorem{ex}[thm]{Example}

\newtheorem{defi}[thm]{Definition}
  
\theoremstyle{remark}
\newtheorem{rem}{Remark}[section]

\newtheorem*{acknowledgment}{Acknowledgments}

\newcommand{\atil}{\widetilde{\alpha}}

\title{Refined Dijkgraaf-Witten invariant of spin 3-manifolds}

\author{Serban Matei Mihalache}%
\address{Department of Mathematics, Tohoku University, 
6-3, Aoba, Aramaki-aza, Aoba-ku, 
Sendai, 980-8578, Japan}
\email{matei.mihalache.q3@dc.tohoku.ac.jp}

\begin{document}1
\maketitle

\begin{abstract}
We give a construction of a state sum invariant of a closed spin 3-manifold based on a super 3-cocycle $(\atil, \omega)$ and a combinatorial representation of a spin 3-manifold, where $\omega$ is a $\mathbb{Z}_2$-valued cocycle and $\atil$ is a 3-cochain satisfying a 3-cocycle condition with a sign coming from the 2-cocycle $\omega$.
The definition of the invariant is similar to the state sum construction of the Dijkgraaf-Witten invariant, except it uses the spin structure to take care of the sign in the 3-cocycle condition.
We also give an example of the invariant and see that it is sensitive to the spin structure.
\end{abstract}

%\tableofcontents

\section{Introduction}
In the '90s, R. Dijkgraaf and E. Witten \cite{DW} introduced a topological invariant of an oriented 3-manifold $M$, now called Dijkgraaf-Witten invariant based on $\text{U}(1)$-valued 3-cocycle $d$ of a finite group $G$:
\begin{align*}
    \text{Z}_{\alpha}(M)&=\frac{1}{|G|}\sum_{f\in\text{Hom}(\pi_1(M),G)}<f^*([\alpha]),[M]>,
\end{align*}
where $[\alpha]$ is a cohomology class in $H^3(G;\text{U}(1))$, and $[M]$ is a fundamental class of $M$.
Essentially, this invariant can be thought of as looking at the characteristic class of all the principal $G$-bundle over $M$ defined by $[\alpha]$.
Although the definition of the invariant is easy, trying to calculate based on this definition is not simple.
In the same paper, they also suggested a construction based on the state sum model, which was later proved to be an invariant \cite{Wakui}.

%There are many generalizations of Dijkgraaf-Witten invariant \cite{}.

In this paper we consider a generalization of the Dijkgraaf-Witten invariant using a super 3-cocycle $(\atil,\omega)$ of finite group $G$ introduced by Z.-C. Gu and X.-G. Wen in \cite{GW} where $\atil$ and $\omega$ are 3-cochain and 2-cocycle respectively, satisfying the following equation:
\begin{align}
    \atil(g,h,k)\,\atil(g,hk,l)\,\atil(h,k,l)&=(-1)^{\omega(g,h)\omega(k,l)}\atil(gh,k,l)\,\atil(g,h,kl),
\end{align}
for all $g,h,k,l\in G$.
A similar construction was already proposed by D. Gaiotto and A. Kapustin \cite{GK} where they introduced the Gu-Wen Grassmann integral (see also \cites{Tata, Kob}) to define invariants of spin 3-manifold using super 3-cocycles, or more generally super fusion categories.
\\

Here, we give a construction based on the combinatorial representation of closed spin 3-manifolds given by R. Benedetti and C. Petronio, called spin normal o-graphs.
Using this combinatorial representation, the closed spin 3-manifolds can be represented as oriented virtual knots with $\mathbb{Z}_2$ weights on each edges (Definition \ref{def:spin normal o-graph}).
In order to define the invariant, we slightly modify the spin normal o-graphs and introduce the planar spin normal o-graphs.
Using the planar normal o-graph, we give a state sum construction using the super 3-cocycle $(\atil,\omega)$.
In \cite{BP}, the authors introduced certain moves on spin normal o-graph and showed that two spin normal o-graphs representing the same closed spin 3-manifold are connected by finite sequence of these moves.
Modifying these moves to planar spin normal o-graphs, we show that the state sum construction is indeed invariant under the moves, thus showing that this is an invariant of closed spin 3-manifolds.
We also give an example of the invariant and see that the invariant is indeed sensible to the spin structure.

\begin{acknowledgment}
 I would like to thank Y. Terashima and S. Suzuki for valuable discussions. This work is supported by JSPS KAKENHI Grant Number JP 22J11429.
\end{acknowledgment}

\section{Calculus of spin 3-manifold}\label{sec:Calculus of spin 3-manifold}
In order to define the state sum, we need a combinatorial description of spin 3-manifolds.
There are few ways to do this \cites{Saw, KR, BR}.
In this paper, we make use of the one given by R. Benedetti and C. Petronio in \cite{BP}, which is based on branched ideal triangulations (or branched standard spines in terms of dual perspective).

\subsection{Spin normal o-graph}\label{subsec:Spin normal o-graph}

\begin{defi}[\cite{BP}]\textbf{}
\label{def:bp}
A \textbf{normal o-graph} is an oriented virtual link diagram, i.e., 
a finite connected $4$-valent graph $\Gamma$ immersed in $\mathbb{R}^2$ with the following conditions:
\begin{spacing}{0.5}
\end{spacing}
\begin{description}
\setlength{\itemsep}{1mm}
\setlength{\parskip}{1mm} 
\item[N1] At each vertex, a sign $+$ or $-$ is indicated, which is represented by the over-under notation as in Figure \ref{fig:tetrahedra_to_crossing}, 
\item[N2] Each edge has an orientation such that it matches among two edges which are opposite to each other at a vertex.
\end{description}
\end{defi}

Since the normal o-graph $\Gamma$ is immersed in $\mathbb{R}^2$, there are two types of crossings: \textbf{true vertex} and the \textbf{fake crossing}, where the true vertex, which is depicted as a crossing with a distinguished dot in the middle, comes from the original vertex of the graph, and the fake crossing comes from the immersion.

Given a normal o-graph $\Gamma$, one can canonically construct an oriented $3$-manifold $M(\Gamma)$ as follows.
We fix an orientation of $\mathbb{R}^3$ and place the normal o-graph on $\mathbb{R}^2\subset \mathbb{R}^3$.
Then we replace each of its vertices with a tetrahedron (with the orientation given by $\mathbb{R}^3$), and glue the faces of the ideal tetrahedra. The way to glue the faces of ideal tetrahedra is specified by the order of vertices of ideal tetrahedra defined as in Figure \ref{fig:tetrahedra_to_crossing}, i.e.,
we glue faces by the orientation reversing map which preserves the order of vertices.
We write the resulting triangulation by $T$.
Deleting all the regular neighborhoods of the vertices of $T$, we get the 3-manifold $M(\Gamma)$ with non-empty boundary, and $T$ is the ideal triangulation of $M(\Gamma)$.
Note that this ideal triangulation has an additional structure called \textbf{branching}, which is a choice of orientation of edges such that for every face of the ideal triangulation, the edge orientation is not cyclic.
A tetrahedron in a branched ideal triangulation is called positive/negative if it is homeomorphic to the left/right hand side of branched tetrahedra in Figure \ref{fig:tetrahedra_to_crossing} i.e., positive/negative if it corresponds to the true vertex of type $+$/$-$.

\begin{figure}[H]
    \centering
    \includegraphics{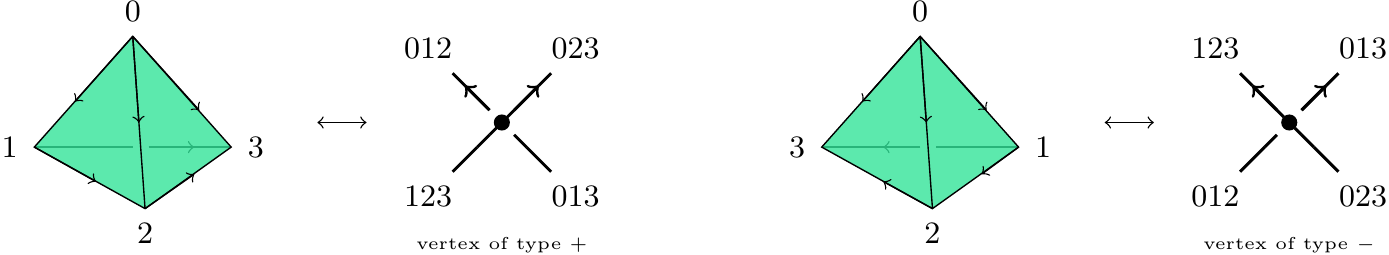}
    \caption{The correspondence between branched tetrahedra and crossing.}
    \label{fig:tetrahedra_to_crossing}
\end{figure}

\begin{defi}[\cite{BP}]
    A normal o-graph $\Gamma$ is called closed if $\Gamma$ satisfies the following conditions:
\begin{description}
\setlength{\itemsep}{1mm}
\setlength{\parskip}{1mm} 
    \item[C1] If one removes the true vertices of $\Gamma$ and joins the edges which are opposite to each other, the result is a unique oriented circuit,
    \item[C2] Ideal triangulation $T$ constructed from $\Gamma$ has a single vertex,
    \item[C3] The number of edges of $T$ is exactly one more than the number of tetrahedra in $T$.
\end{description}
\end{defi}

The conditions on closed normal o-graph $\Gamma$ ensure that the resulting manifold $M(\Gamma)$ is an ideally triangulated 3-manifold with $S^2$ boundary, thus we can cap off the boundary by adding $B^3$ and the result is an closed oriented 3-manifold.

Consider the union of closed oriented circuits (with dots) associated with the underline closed normal o-graph $\Gamma$ by the rule in Figure \ref{fig:heegaard diagram}.
For any closed circuit $\gamma$, there exist an oriented loop $e_1\cdots e_p$ in $\Gamma$ which corresponds to $\gamma$.
Here, $e_1,\cdots ,e_p$ represent edges of $\Gamma$, and there may be some duplication among them.

\begin{defi}[\cite{BP}]\label{def:spin normal o-graph}
A \textbf{spin normal o-graph} $\Gamma$ is a closed normal o-graph with $\mathbb{Z}_2$ weights on each edge, satisfying the following condition:
\begin{spacing}{0.5}
\end{spacing}
\begin{description}
\setlength{\itemsep}{1mm}
\setlength{\parskip}{1mm} 
\item[S] The following process should occur for each circuit $\gamma$ featuring solid dots in Figure \ref{fig:heegaard diagram}. Let $e_1\cdots e_p$ denote the loop in $\Gamma$ that corresponds to $\gamma$. For each edge $e_i$, let $x_i\in\mathbb{Z}_2$ be the weight attached to the edge $e_i$. If there are $m$ solid dots on $\gamma$, then the equation $\sum_{i=1}^p x_i\equiv \frac{m}{2}+1\mod 2$ should hold.
\end{description}
\end{defi}

Given a spin normal o-graph $\Gamma_s$, the above mentioned construction produces a closed oriented 3-manifold $M(\Gamma_s)$ with $\mathbb{Z}_2$ weights producing a trivialization of the tangent bundle $TM(\Gamma_s)$ over the 1-skeleton of the dual polyhedron of $T$.
The condition \textbf{S} of the spin normal o-graph ensures that this trivialization extends over the 2-skeleton, thus defining a spin structure.
For more details see \cite{BP}*{Chapter 7}.

It is shown in \cite{BP} that every closed spin 3-manifold can be represented by a spin normal o-graph and two spin normal o-graphs are connected by a finite sequence of local moves depicted in Figures \ref{fig:Reidemeister_type_moves} to \ref{fig:Z_2_branched_MP_move} if and only if they define the same spin 3-manifold.

\subsection{Planar spin normal o-graph}\label{subsec:planar spin normal o-graph}
Let $\Gamma$ be a normal o-graph.
A planar normal o-graph is a $\Gamma$ immersed in $\mathbb{R}^2$ such that there are no inflection points with respect to $y$-axis and all the orientations of the true vertices are going upwards.
For an edge $e$ of a planar normal o-graph, let $w(e) \in \mathbb{Z}$ be the winding number of an edge, which is calculated as sums of fractions $\frac{1}{2}$ or $-\frac{1}{2}$ attached to the maximal/minimal point on the edge $e$, see Figure \ref{fig:winding_number}.

\begin{figure}[H]
  \begin{minipage}[b]{0.45\linewidth}
    \centering
    \includegraphics{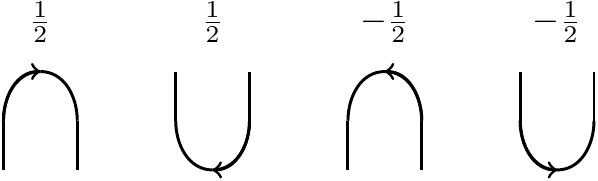}
    \caption{The winding number}
    \label{fig:winding_number}
  \end{minipage}
  \begin{minipage}[b]{0.45\linewidth}
    \centering
    \includegraphics[keepaspectratio, scale=1]{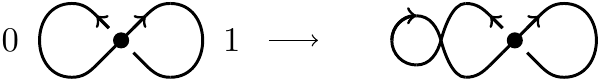}
    \caption{Example of planar spin normal o-graph}
    \label{fig:Example of planer spin normal o-graph}
  \end{minipage}
\end{figure}

Now, let us consider a spin normal o-graph $\Gamma_s$.
Recall that in this case, each edge $e$ of $\Gamma_s$ has $\mathbb{Z}_2$ weights $z(e)$ attached on.
A \textbf{planar spin normal o-graph} $P$ is a planar normal o-graph of $\Gamma_s$ such that the winding number modulo 2 of each edge is matching with $\mathbb{Z}_2$ weights, i.e., $w(e)\equiv z(e)$ mod 2 for every edge $e\in\Gamma_s$.
Given a spin normal o-graph $\Gamma_s$ (up to moves in Figure \ref{fig:Reidemeister_type_moves}), the planar normal o-graph is not unique, but for two planar normal o-graph of $\Gamma_s$ there always exist a finite sequence of Reidemeister type moves (which are depicted in Figure \ref{fig:planer_reidemeister_type_move}) connecting the one to the other.

\begin{figure}[H]
    \centering
    \includegraphics{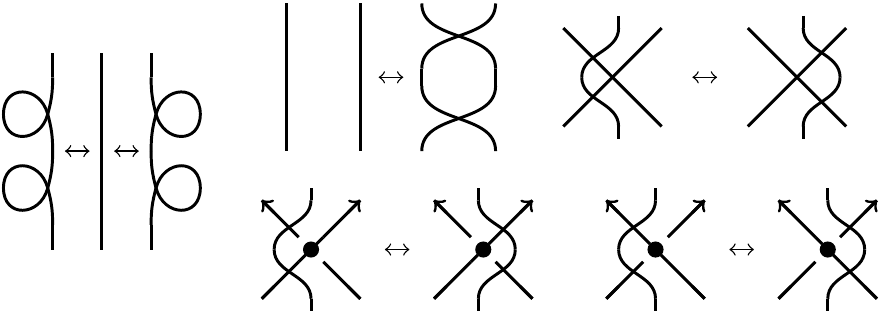}
    \caption{Reidemeister type moves of planar spin normal o-graph. The unoriented edges should be oriented so that the orientation matches before and after the move.}
    \label{fig:planer_reidemeister_type_move}
\end{figure}

\section{Invariant}\label{sec:Invariant}
\subsection{Super 3-cocycle}\label{subsec:Super 3-cocycle}

Throughout the paper, we let $G$ be a finite group and $\mathbb{K}$ a field.

\begin{defi}
A $\mathbb{K}^*$-valued \textbf{super 3-cocycle} of $G$ is a pair $(\atil,\omega)$, where $\atil$ and $\omega$ are maps
$\atil\colon G^{\times 3}\to \mathbb{K}^*$, $\omega\colon G^{\times 2}\to \mathbb{Z}_2$
which satisfies
\begin{align}
    \omega(g,h)+\omega(gh,k)&=\omega(h,k)+\omega(g,hk) \label{eq:2-cocycle}, \\
    \atil(g,h,k)\,\atil(g,hk,l)\,\atil(h,k,l)&=(-1)^{\omega(g,h)\omega(k,l)}\atil(gh,k,l)\,\atil(g,h,kl) \label{eq:3-supercocycle},
\end{align}
for all $g,h,k,l\in G$.
\end{defi}

When $\omega(g,h)$ equals $0$ for every $g,h\in G$, the super 3-cocycle is the same as the ordinary $\mathbb{K}^*$-valued 3-cocycle of a group $G$.
\begin{comment}
We call the $(\atil,\omega)$ \textbf{normalized} if $\atil(1,g,h)=\atil(g,1,h)=\atil(g,h,1)=1$ and $\omega(1,g)=\omega(g,1)=0$ for every $g,h\in G$.
\end{comment}

\begin{ex}\label{ex:3-supercocycle of cyclic group}
For a cyclic group $\mathbb{Z}_n$, let $\omega$ be a 2-cocycle:
\begin{align*}
\omega(a,b)=
  \begin{cases}
    0 & \text{if $a+b<n$,} \\
    1 & \text{if $a+b\geq n$.}
  \end{cases}
\end{align*}

Let $\atil\colon\mathbb{Z}_n\to\mathbb{C}^*$ be
\begin{align*}
    \atil(a,b,c)=\text{exp}(\frac{\pi i}{n}\omega(a,b)c)
\end{align*}
Then $(\atil,\omega)$ is a $\mathbb{C}^*$-valued super 3-cocycle of $\mathbb{Z}_n$.

\end{ex}

For more details on super cocycles, see for example \cite{GJ}.

\subsection{Definition of invariant}\label{subsec:Definition of invariant}
Let $\mathbb{K}$ be a field, $G$ a finite group and $(\atil, \omega)$ a $\mathbb{K}^*$-valued super 3-cocycle.  
For a spin 3-manifold $(M,s)$, let $P$ be a planar spin normal o-graph representing $(M,s)$ and $T$ its branched ideal triangulation constructed from $P$.
Let $E$ be the set of oriented edges of $T$.
\begin{defi}
    A map $\phi\colon E\to G$ is called an edge coloring.
    Furthermore, an edge coloring is called admissible if for every face $F$ of $T$ with edges $e_1$, $e_2$ and $e_3$, and orientation matching with the ones in Figure \ref{fig:Oriented edge of face $F$.}, the following holds:
\begin{align*}
    \phi(e_1)\phi(e_2)\phi(e_3)^{-1}=1.
\end{align*}
\end{defi}

For positive and negative tetrahedra, the admissible coloring is determined by the coloring of the edges $e_{01}$, $e_{12}$ and $e_{23}$.
For example, in case of the positive tetrahedra in the left hand side of Figure \ref{fig:Admissible coloring of tetrahedra}, the edge coloring should satisfy: $\phi(e_{02})=gh$, $\phi(e_{13})=hk$ and $\phi(e_{03})=ghk$ where $\phi(e_{01})=g$, $\phi(e_{12}) = h$ and $\phi(e_{23}) = k$.

\begin{figure}[H]
  \begin{minipage}[b]{0.45\linewidth}
    \centering
    \includegraphics[keepaspectratio, scale=1.3]{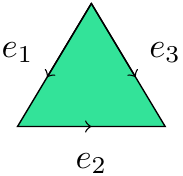}
    \caption{Oriented edges of $F$.}
    \label{fig:Oriented edge of face $F$.}
  \end{minipage}
  \begin{minipage}[b]{0.45\linewidth}
    \centering
    \includegraphics[keepaspectratio, scale=1]{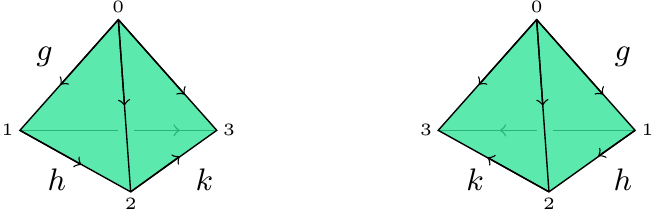}
    \caption{Admissible coloring of tetrahedron.}
    \label{fig:Admissible coloring of tetrahedra}
  \end{minipage}
\end{figure}

For admissible coloring $\phi$, let us define two scalars $W(T;\phi)$ and $\theta(P;\phi)$ where the invariant will be defined as a sum of $\theta(P;\phi)
W(T;\phi)$ over all the admissible colorings.
In order to keep the notation simple, when there is no confusion, we will not make the difference between edges and the coloring of the edges given by $\phi$.

\vskip\baselineskip
\noindent
\textbf{Definition of $W(T;\phi)$:}

For positive tetrahedra $\sigma$ in $T$ with admissible coloring given by $\phi$, we define $W(\sigma;\phi)$ as $\atil(g,h,k)$, where the specified edges $g$, $h$ and $k$ are chosen canonically using branching structure (see the left hand side of Figure \ref{fig:Admissible coloring of tetrahedra}).
For the negative tetrahedra, let $W(\sigma;\phi)$ to be $\atil(g,h,k)^{-1}$, where again the specified edges are chosen canonically from the branching structure, see the right hand side of Figure \ref{fig:Admissible coloring of tetrahedra}.
Then the scalar $W(T;\phi)$ is defined by
\begin{align*}
W(T;\phi) := \prod_{\sigma\in T}W(\sigma;\phi).
\end{align*}

\vskip\baselineskip
\noindent
\textbf{Definition of $\theta(P;\phi)$:}

Let $c$ be a fake crossing of $P$.
Recall from the construction given in Subsection \ref{subsec:Spin normal o-graph} that every edge of $P$ corresponds to a face of $T$.
Let $F_1$ and $F_2$ be faces corresponding to each edge of crossing $c$.
Then let $\theta(c\,;\phi)$ be $(-1)^{\omega(g,h)\,\omega(k,l)}$, 
where edges $g$, $h$, $k$ and $l$ are specific edges of faces $F_i$ determined by the branching structure of $T$, see the figure below.

\begin{figure}[H]
    \centering
    \includegraphics{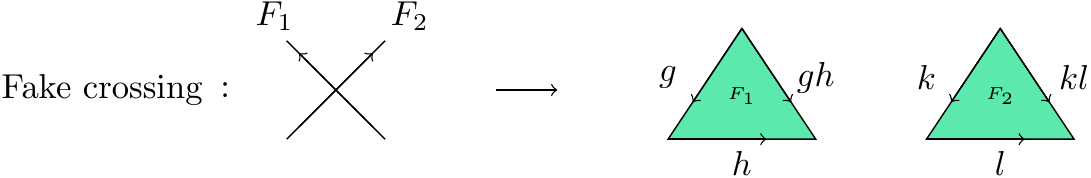}
\end{figure}
Then the $\theta(P;\phi)$ is defined as products of $\theta(c;\phi)$ over every fake crossing:
\begin{align*}
\theta(P;\phi) := \prod_{c\text{\,:\,fake crossing}}\theta(c\,;\phi).
\end{align*}

Finally, we let $Z_{\atil,\omega}(M,s)$ be the sums of product of $W(T;\phi)$ and $\theta(P;\phi)$ over the admissible coloring of $T$:
\begin{align*}
    Z_{\atil,\omega}(M,s) = \sum_{\phi}\theta(P;\phi)W(T;\phi).
\end{align*}

\subsection{Invariance of $Z_{\atil,\omega}(M,s)$}\label{subsec:Invariance of invarinat}

\begin{thm}
$Z_{\atil,\omega}(M,s)$ is an invariant of closed spin 3-manifolds.
\end{thm}

\begin{proof}
In order to prove the theorem, let us rewrite the definition in a more suitable form for the proof.
Given a planar spin normal o-graph $P$, replace each vertex and edge with the circuits, as shown in the figure below.
Note that if an edge connects different types of true vertices, we add a twist in the first and second circuits.

\begin{figure}[H]
    \centering
    \includegraphics[scale=0.95]{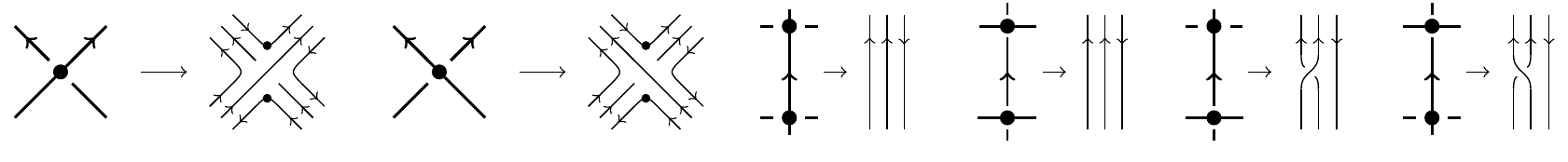}
    \caption{Union of circuits associated to normal o-grpah.}
    \label{fig:heegaard diagram}
\end{figure}

The replaced diagram is a collection of closed oriented circuits and it will be denoted as $E(P)$.
Geometrically, these circuits represent the boundary of 2-cells of the dual polyhedron $T^*$.
From this point of view, the admissible coloring $\phi\colon E\to G$ gives an ``admissible'' coloring $\phi\colon E(P)\to G$.

Let us define what an admissible coloring of $E(P)$ is.
In the case of $E$, we used the branching structure of $T$ to define the ordering of the edges of the face in order to define the admissible conditions.
We do the same for the $E(P)$, but here the reader should be careful about how the circuits are ordered.
The ordering of the three circuits for a replaced edge depends on the type of the vertex an edge is going towards or coming from; if the edge is going towards or coming from a vertex of type $+$, the ordering of the three circuits near the type $+$ vertex is from left to right.
If the edge is going towards or coming from a vertex of type $-$, the ordering of the three circuits near the type $-$ vertex is: the first one is the middle circuit, the second one is the left-most circuit and the third one is the right-most circuit.
Note that if the edge of $P$ connects different types of vertices, there was a twist between the left and middle circuits, thus preserving the ordering.

Using this ordering, the admissible coloring is a map $\phi\colon E(P)\to G$ satisfying the following condition for each replaced edge of $P$:

\begin{figure}[H]
    \centering
    \includegraphics{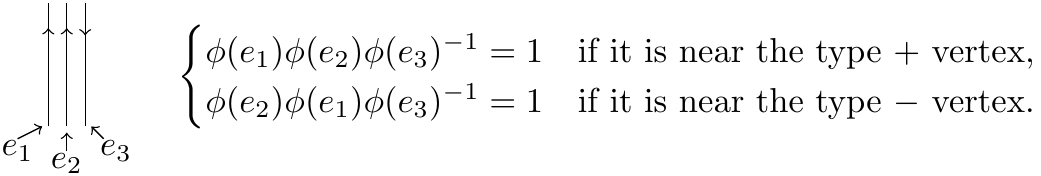}
\end{figure}

By using the admissible coloring $\phi$ of the $E(P)$, the scalar $\theta(P;\phi)W(T;\phi)$ is redefined as follows: for each vertex of type $+$, associate the scalar $\atil(g,h,k)$; for type $-$, associate the scalar $\atil(g,h,k)^{-1}$; and for each fake crossing associate the scalar $(-1)^{\omega(g,h)\omega(k,l)}$ and take the product of all the associated scalars (see the figure below).
Here. for the scalar $(-1)^{\omega(g,h)\omega(k,l)}$, the input of the $\omega(-,-)$ should match the ordering of the circuits (the scalar for fake crossing in the figure below is for edges near the type $+$ vertex).

\begin{figure}[H]
    \centering
    \includegraphics{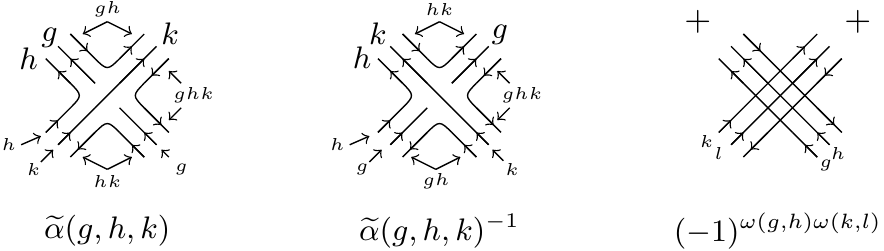}
\end{figure}

Now, let us show the invariance of $Z_{\atil,\omega}(M,s)$.
Recall that the two spin normal o-graphs defining the same spin 3-manifold are connected by the moves in Figures \ref{fig:Reidemeister_type_moves} to \ref{fig:Z_2_branched_MP_move}.
Since the invariant uses the planar spin o-graph, we need to represent the moves in Figures \ref{fig:H-moves} to \ref{fig:Z_2_branched_MP_move} in a planar way.
More formally, let $\Gamma_s$ and $\Gamma_s^{\prime}$ be spin normal o-graphs, where $\Gamma_s^{\prime}$ is obtained from the $\Gamma_s$ with one of the moves in Figures \ref{fig:H-moves} to \ref{fig:Z_2_branched_MP_move} applied.
Let $P$ and $P^{\prime}$ be planar spin o-graphs of $\Gamma_s$ and $\Gamma_s^{\prime}$.
Then, ``the planar representation of the move" means that we need to represent the move up to moves in Figure \ref{fig:planer_reidemeister_type_move}, so that the winding number modulo 2 of the edges of the $P$ and $P^{\prime}$ matches with the $\mathbb{Z}_2$ weights of the $\Gamma_s$ and $\Gamma_s^{\prime}$.
Then, two planar spin normal o-graphs defining the same spin 3-manifold are connected by the moves in Figure \ref{fig:planer_reidemeister_type_move} and the planar representation of the moves in Figures \ref{fig:H-moves} to \ref{fig:Z_2_branched_MP_move}.

First, let us show that for a fixed admissible coloring $\phi$, the scalar $\theta(P;\phi)$ is invariant under the moves in Figures \ref{fig:planer_reidemeister_type_move} and \ref{fig:H-moves}.
The invariance of all the moves in Figure \ref{fig:planer_reidemeister_type_move}, except the ones involving a true vertex, is straightforward to see.
Let us show the invariance of the Reidemeister 3 type move involving type $+$ crossing.
As usual, we show the invariance of a local scalar before and after the move.

\begin{figure}[H]
    \centering
    \includegraphics{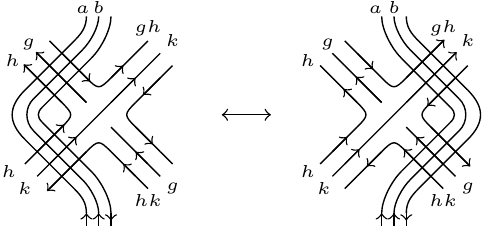}
\end{figure}

\begin{align*}
    (-1)^{\omega(a,b)\omega(g,h)}(-1)^{\omega(a,b)\omega(h,k)} &= (-1)^{\omega(a,b)(\omega(g,h)+\omega(h,k))}
    =(-1)^{\omega(a,b)(\omega(gh,k)+\omega(g,hk))}\\
    &=(-1)^{\omega(a,b)\omega(gh,k)}(-1)^{\omega(a,b)\omega(g,hk)}\\
\end{align*}
\noindent
where we used the 2-cocycle condition \eqref{eq:2-cocycle} of $\omega$ for the second equality.

For the H-moves in Figure \ref{fig:H-moves}, the planar representation of this move is 

\begin{figure}[H]
    \centering
    \includegraphics{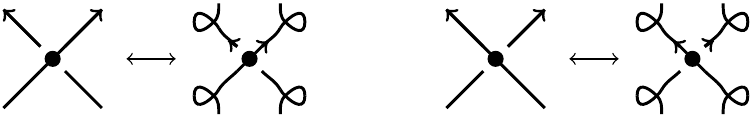}
\end{figure}
\noindent
The invariance of the H-moves again follows from the 2-cocycle condition \eqref{eq:2-cocycle} of $\omega$.

Now, let us see the invariance of $\theta(P;\phi)W(T;\phi)$ for the local moves in Figure \ref{fig:Z_2_branched_0-2_move} to Figure \ref{fig:Z_2_branched_MP_move}.
For these moves, given an admissible coloring $\phi$ of the closed circuits $E(P)$, there is a unique admissible coloring $\phi^{\prime}$ of $E(P^{\prime})$, where $P^{\prime}$ is a planar spin normal o-graph after the move, so that the $\phi^{\prime}$ defines the same coloring one the closed circuits coming from the $E(P)$, and all the admissible coloring of $E(P^{\prime})$ are such.
Thus, we prove the invariance of $\theta(P;\phi)W(T;\phi)$ for a fixed admissible coloring $\phi$.

For the branched 0-2 move in Figure \ref{fig:Z_2_branched_0-2_move}, the planar representation of the move and the admissible coloring is:

\begin{figure}[H]
    \centering
    \includegraphics[scale=1]{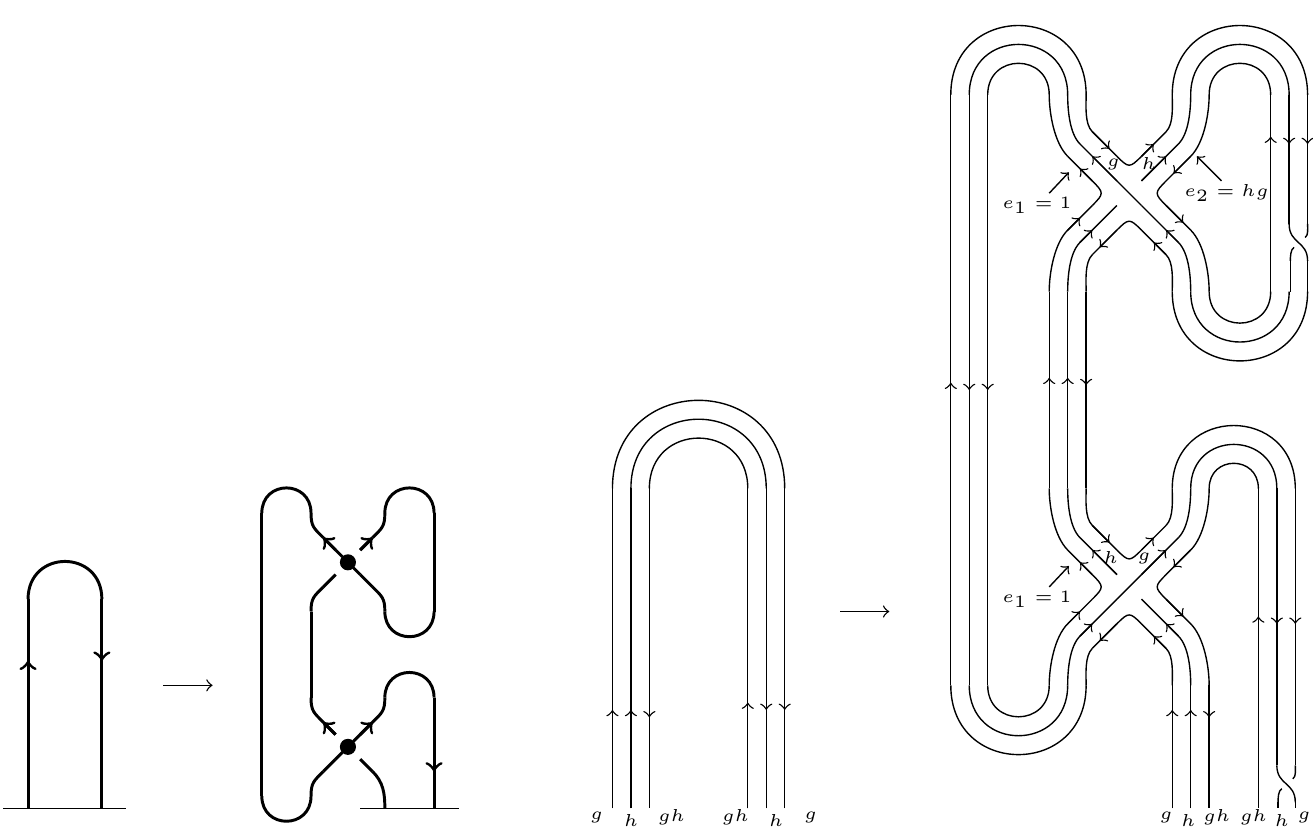}
\end{figure}
\noindent

Here, after the move, there are two new closed circuits, $e_1$ and $e_2$.
If we assume that the admissible coloring $\phi^{\prime}$ is induced from the coloring before the move, the colorings of the circuits are given by $\phi^{\prime}(e_1)=1$ and $\phi^{
\prime}(e_2)=hg$.
In this case, the planar representation does not have any fake crossings, and locally the scalar associated before and after the move is $1$ and $\atil(h,1,g)\,\atil(h,1,g)^{-1}$ which are equal.
Thus, the scalar $\theta(P;\phi)W(T;\phi)$ is invariant under the branched 0-2 move.

Next, let us check the moves in Figure \ref{fig:Z_2_branched_MP_move} with a fixed $\phi$.
We observe that even before the planar representation, the moves carry a graph with two true vertices to a graph with three true vertices with a fake crossing that exactly matches the condition of a super 3-cocycle \eqref{eq:3-supercocycle}.
In order to see the invariance, we first need a planar representation of the moves, which might add new fake crossings.
But in the case of the moves in Figure \ref{fig:Z_2_branched_MP_move}, possibly applying H-move in Figure \ref{fig:H-moves}, it can be checked that the planar representation does not contain any fake crossings except the one coming from the original move.
For example, the planar representation of the move in the top left-hand side of Figure \ref{fig:Z_2_branched_MP_move} is:

\begin{figure}[H]
    \centering
    \includegraphics[scale=0.8]{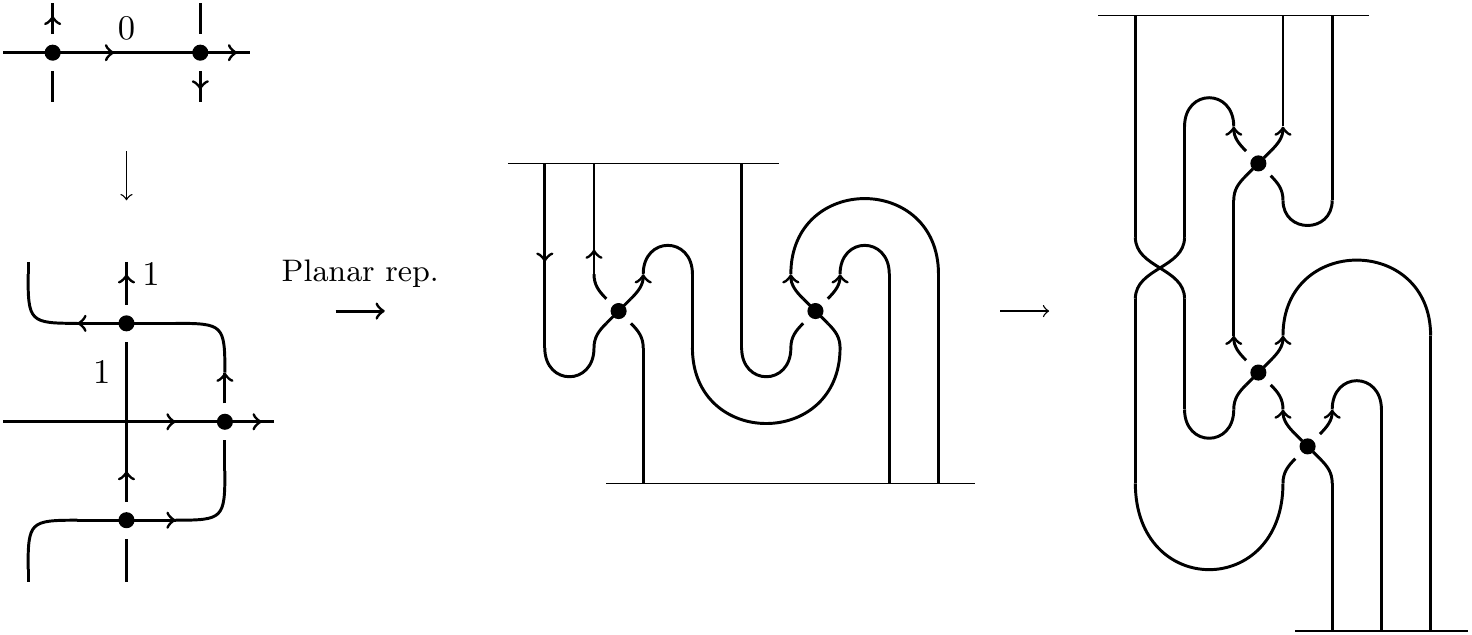}
\end{figure}
\noindent
Thus, we see that indeed there are no fake crossings except the one coming from the original move.

The admissible coloring of this particular move is given by:

\begin{figure}[H]
    \centering
    \includegraphics{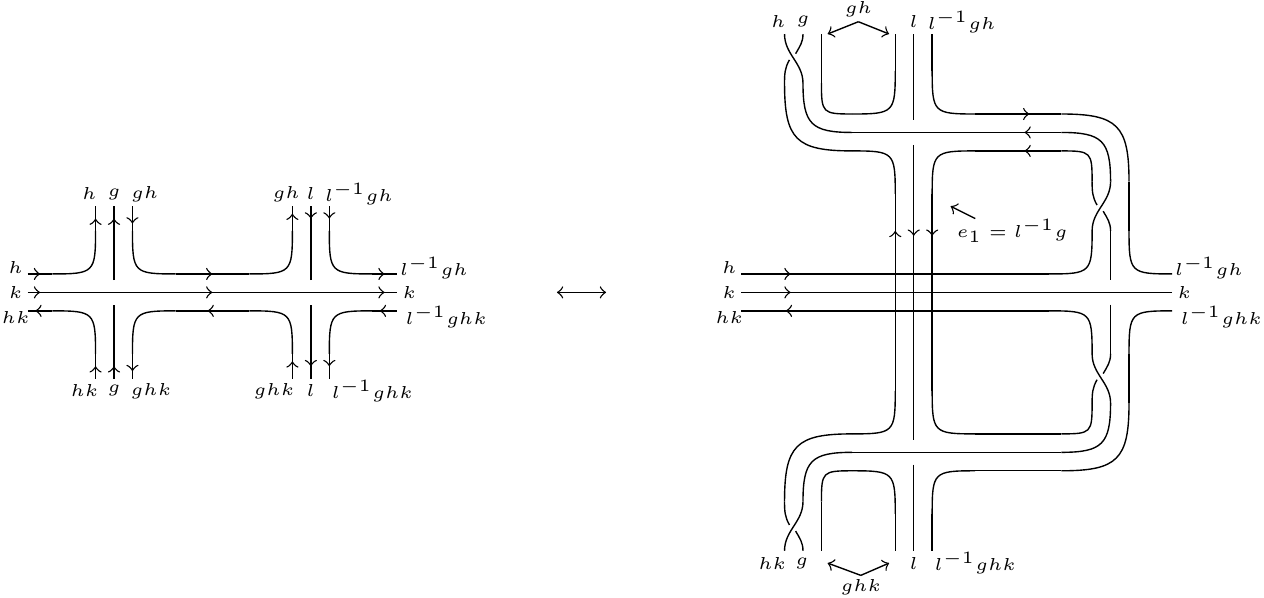}
\end{figure}
\noindent
The local scalars before and after the move are:
\begin{align*}
    &\atil(g,h,k)\,\atil(l,l^{-1}gh,k)^{-1} \\
    &(-1)^{\omega(l,l^{-1}g)\,\omega(h,k)}\,\atil(l,l^{-1}g,h)\,\atil(l^{-1}g,h,k)\,\atil(l,l^{-1}g,hk)^{-1}
\end{align*}
Thus we need to show that these two scalars are the same, i.e.,
\begin{align}
    \atil(g,h,k)\,\atil(l,l^{-1}gh,k)^{-1}&=(-1)^{\omega(l,l^{-1}g)\,\omega(h,k)}\,\atil(l,l^{-1}g,h)\,\atil(l^{-1}g,h,k)\,\atil(l,l^{-1}g,hk)^{-1}\\
    \iff \qquad \atil(g,h,k)\,\atil(l,l^{-1}g,hk)&=(-1)^{\omega(l,l^{-1}g)\,\omega(h,k)}\,\atil(l,l^{-1}gh,k)\,\atil(l,l^{-1}g,h)\,\atil(l^{-1}g,h,k)\label{eq:MP-move}
\end{align}
If we set $g\rightarrow l$, $h\rightarrow l^{-1}g$, $k\rightarrow h$ and $l\rightarrow k$ for the super 3-cocycle condition \eqref{eq:3-supercocycle}, the resulting equation is the same as \eqref{eq:MP-move}, thus the equality of \eqref{eq:MP-move} indeed holds.
This shows that the scalar $\theta(P;\phi)W(T;\phi)$ is invariant under a branched MP-move.
Other moves in the Figure \ref{fig:Z_2_branched_MP_move} can be checked in the same way, and all the equalities of the local scalars of the moves are equivalent to the super 3-cocycle condition \eqref{eq:3-supercocycle}.

Finally, let us check the spin CP-move in Figure \ref{fig:Spin_CP_moves}.
The planar representation of the right-hand side of the spin CP-moves is:

\begin{figure}[H]
    \centering
     \includegraphics[scale=0.8]{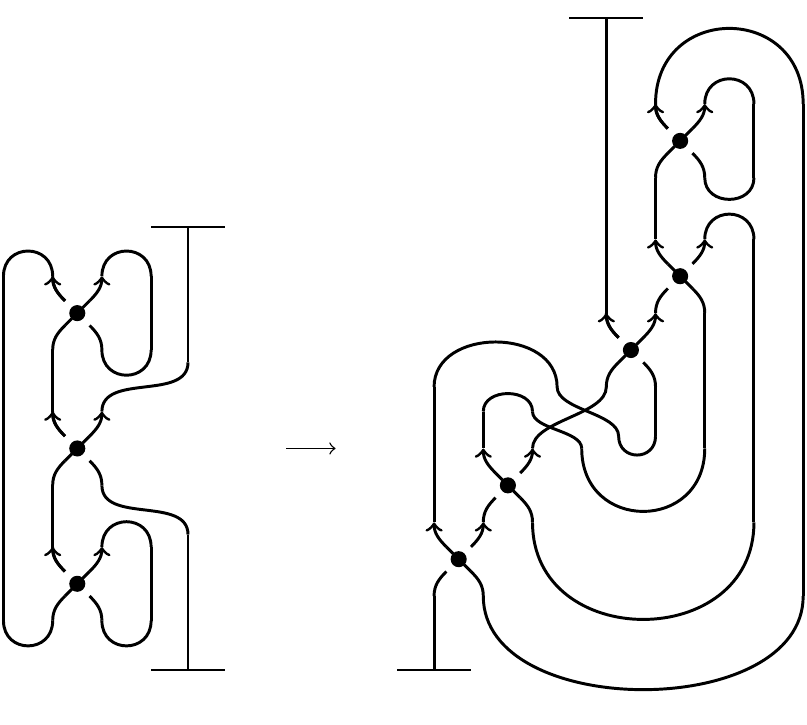}
\end{figure}
\noindent
The admissible coloring is given by:
\begin{figure}[H]
    \centering
    \includegraphics{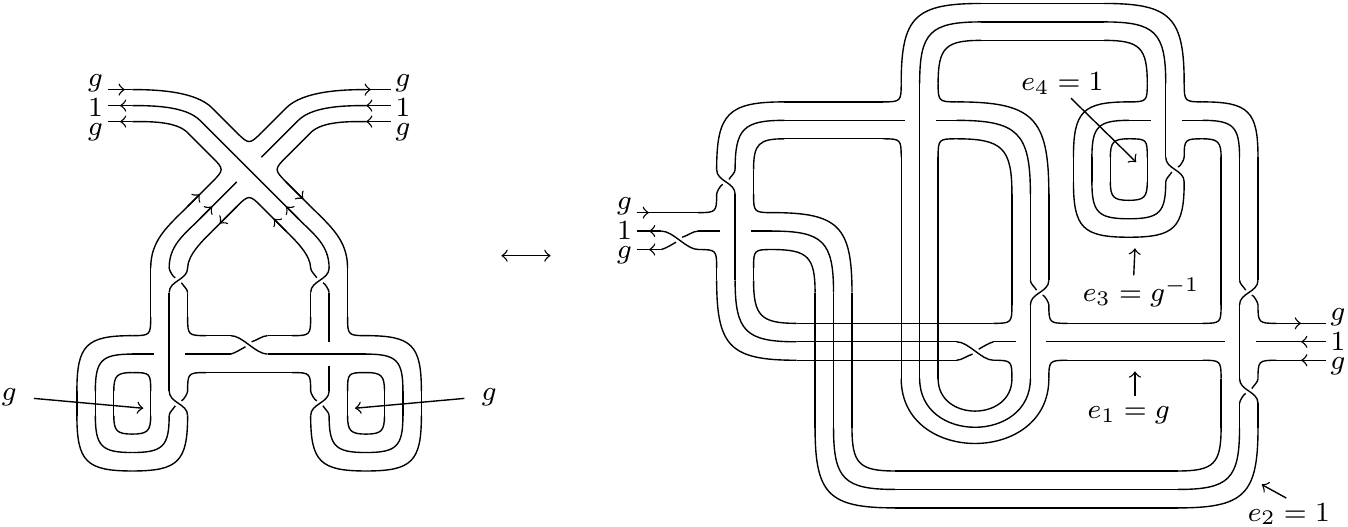}
\end{figure}
Again, after the move, there are 4 new closed circuits $e_1, e_2, e_3$ and $e_4$ whit the canonical coloring given by $g,1,g^{-1}$ and $1$ respectively.
Then, the local tensors are given by:
\begin{align*}
    \text{left hand side:}\quad&\atil(1,g,1)\,\atil(1,1,g),\atil(g,1,1)\\
    \text{right hand side:}\quad&(-1)^{\omega(g,g^{-1})\omega(1,1)+\omega(g,1)\omega(1,1)}\,\atil(1,g,1)^{-1}\,\atil(1,g,g^{-1})^{-1}\,\atil(g,1,1)\,\atil(g,1,g^{-1})^{-1}\,\atil(g,1,g^{-1})\\
    &=(-1)^{\omega(g,g^{-1})\omega(1,1)+\omega(g,1)\omega(1,1)}\,\atil(1,g,1)^{-1}\,\atil(1,g,g^{-1})^{-1}\,\atil(g,1,1).
\end{align*}
If we set $g\rightarrow 1$, $h\rightarrow g$, $k\rightarrow 1$ and $l\rightarrow g$ in the equation \eqref{eq:3-supercocycle}, we get $\atil(1,g,1)=(-1)^{\omega(1,g)\omega(1,g)}$, thus $\atil(1,g,1)=\atil(1,g,1)^{-1}$.
If we set $g\rightarrow 1$, $h\rightarrow 1$, $k\rightarrow g$ and $l\rightarrow g^{-1}$ in the equation \eqref{eq:3-supercocycle}, 
\begin{align*}
    \atil(1,g,g^{-1})^{-1}=(-1)^{\omega(g,g^{-1})\omega(1,1)}\atil(1,1,1)^{-1}\atil(1,1,g).
\end{align*}
Since $\atil(1,1,1)=(-1)^{\omega(1,1)\omega(1,1)}$,
\begin{align*}
    \atil(1,g,g^{-1})^{-1}=(-1)^{\omega(g,g^{-1})\omega(1,1)+\omega(1,1)\omega(1,1)}\atil(1,1,g).
\end{align*}
Thus, the right hand side of the local tensor matches with the left hand side and the scalar is indeed in invariant under the right hand side of spin CP-move.
The invariance under the left hand side of spin CP-move can be checked similarly, thus $Z_{\atil,\omega}(M,s)$ is invariant of closed spin 3-manifold.
\end{proof}

\begin{rem}\label{rem:Z2 weights}
The invariant we defined in this paper can be defined on the (closed) normal o-graph without $\mathbb{Z}_2$ weights.
But in this case, the scalar is not invariant under the R1-type move in Figure \ref{fig:Reidemeister_type_moves} and is only invariant under the framed R1 move in Figure \ref{fig:planer_reidemeister_type_move}.
This suggests that in order to make the scalar to be an invariant, we need to control the number of the ``twists" in the edges, and in our case, the $\mathbb{Z}_2$ weights on the spin normal o-graph are exactly what we need to control these ``twists."
\end{rem}

\subsection{Example of invariant}\label{subsec:Example of spin 3-manifold}

For $p\geq 1$, let $L(p,1)$ be the lens space.
Recall that 
\begin{align*}
H^1(L(p,1);\mathbb{Z}_2)=
    \begin{cases}
        0 & \text{if $p$ is odd,} \\
        \mathbb{Z}_2 & \text{if $p$ is even.}
    \end{cases}
\end{align*}
For any $p$, the following spin normal o-graph with p true vertices represents a spin $L(p,1)$
\begin{figure}[H]
    \centering
    \includegraphics{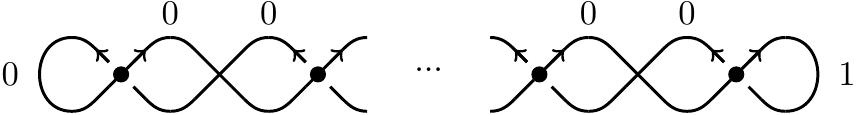}
\end{figure}
Thus, when p is odd, the above diagram represents $L(p,1)$ with a unique spin structure.
When p is even, $L(p,1)$ has two distinct spin structures and the other spin normal o-graphs are given by
\begin{figure}[H]
    \centering
    \includegraphics{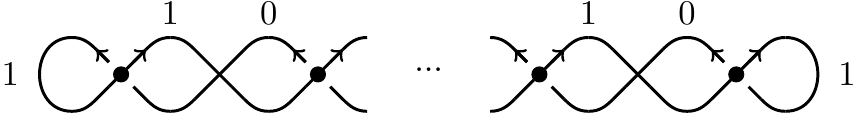}
\end{figure}
Let us consider the case for p=2.
For a super 3-cocycle $(\atil,\omega)$, the invariant is:
\begin{align*}
    Z_{\atil,\omega}(L(2,1),s_1) &= \sum_{g^2=1}(-1)^{\omega(g,g)+\omega(g,1)}\atil(g,1,g)\atil(g,g,g)\\
    Z_{\atil,\omega}(L(2,1),s_2) &= \sum_{g^2=1}\atil(g,1,g)\atil(g,g,g)
\end{align*}
If we set $G=\mathbb{Z}_2$ with the super 3-cocycle given by the Example \ref{ex:3-supercocycle of cyclic group}, the invariant is
\begin{align*}
    Z_{\atil,\omega}(L(2,1),s_1) &=1-i\\
    Z_{\atil,\omega}(L(2,1),s_2) &=1+i
\end{align*}
and we see that the invariant is sensible to the spin structure.

\begin{appendices}
\section{Local moves of spin normal o-graph}\label{sec:Local moves of spin normal o-graph}

Here, we show all the local moves of the spin normal o-graph.

\begin{figure}[H]
    \centering
    \includegraphics{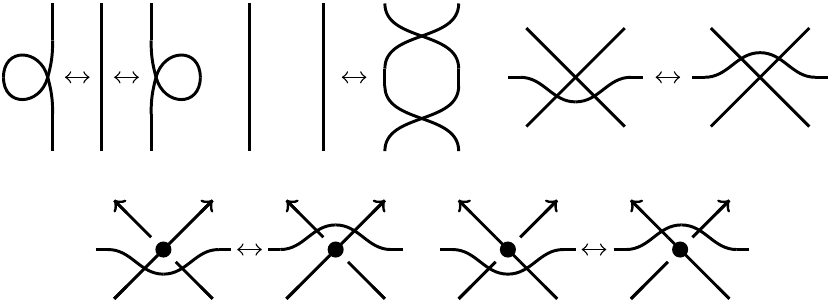}
    \caption{Reidemeister type moves for spin normal o-graph}
    \label{fig:Reidemeister_type_moves}
\end{figure}

\begin{figure}[H]
  \begin{minipage}[b]{0.45\linewidth}
    \centering
    \includegraphics[keepaspectratio, scale=1]{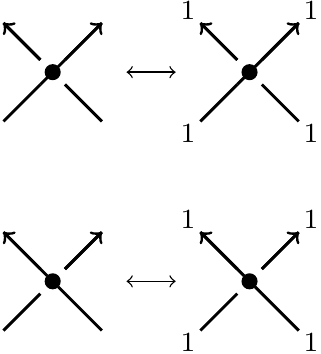}
    \caption{H moves.}
    \label{fig:H-moves}
  \end{minipage}
  \begin{minipage}[b]{0.45\linewidth}
    \centering
    \includegraphics[keepaspectratio, scale=1]{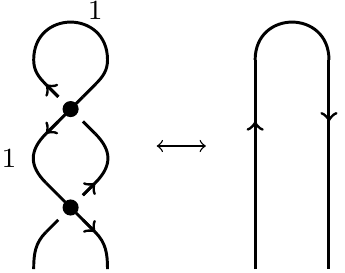}
    \caption{Branched 0-2 move.}
    \label{fig:Z_2_branched_0-2_move}
  \end{minipage}
\end{figure}

\begin{figure}[H]
    \centering
    \includegraphics[scale=0.8]{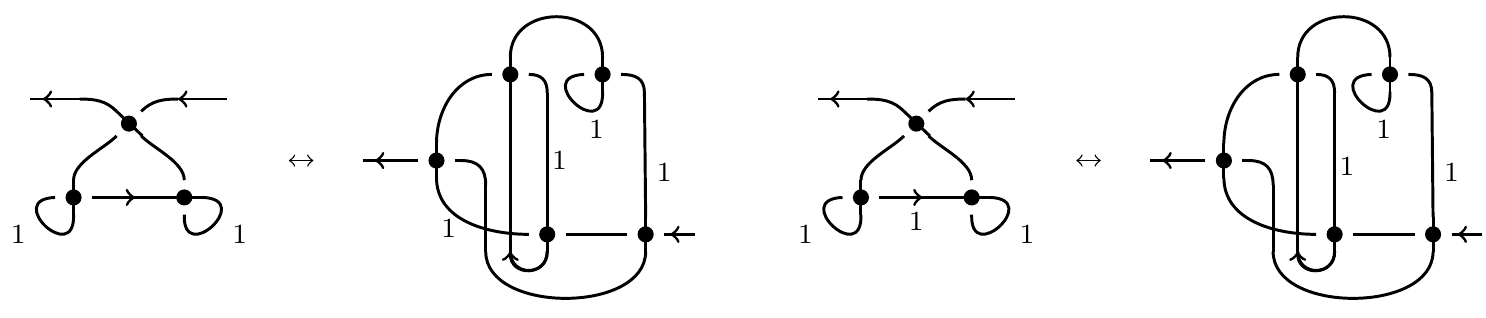}
    \caption{Spin CP-move.}
    \label{fig:Spin_CP_moves}
\end{figure}

\begin{figure}[H]
    \centering
    \includegraphics[scale=0.8]{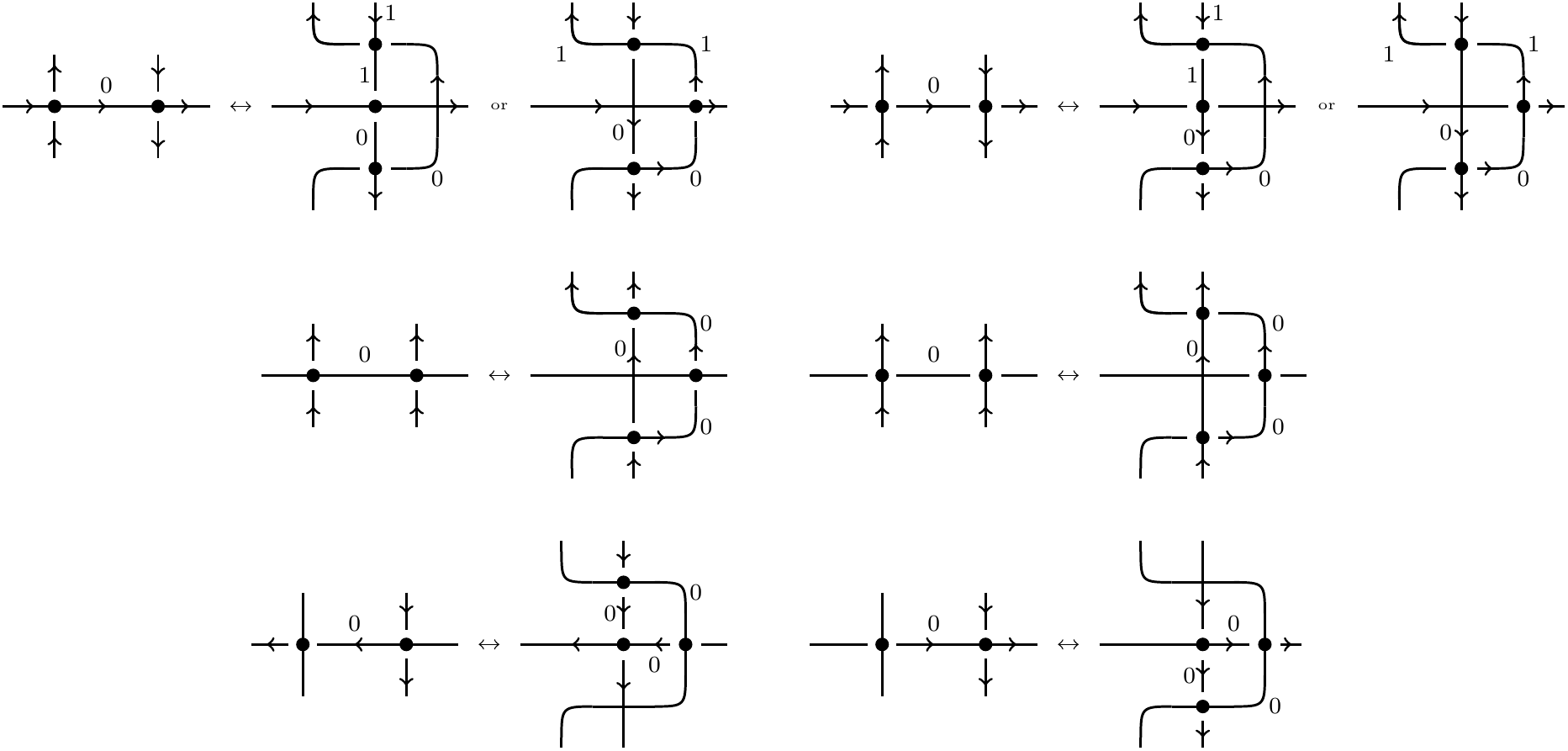}
\end{figure}

\begin{figure}[H]
    \centering
    \includegraphics[scale=0.8]{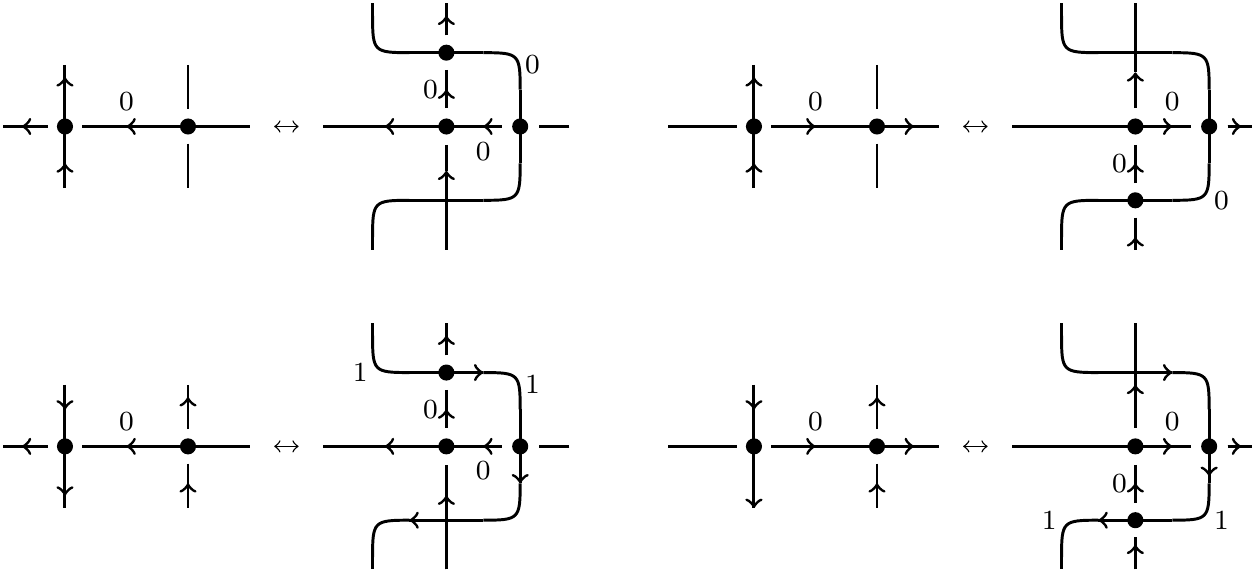}
    \caption{Branched MP-moves. If an edge has multiple weights after the move, the weights should be added in the additive group $\mathbb{Z}_2$.}
    \label{fig:Z_2_branched_MP_move}
\end{figure}

\end{appendices}

\bibliography{references}
\bibliographystyle{plain}

\end{document}